\documentclass[11pt]{amsart}

\author[C.~Sanna]{Carlo Sanna$^\dagger$}
\thanks{$\dagger\,$C.~Sanna is a member of GNSAGA of INdAM and of CrypTO, the group of Cryptography and Number~Theory of the Politecnico di Torino}
\address{\parbox{\linewidth}{
        Department of Mathematical Sciences, Politecnico di Torino\\
        Corso Duca degli Abruzzi 24, 10129 Torino, Italy\\[-8pt]}}
\email{carlo.sanna@polito.it}

\title[Entries of fixed-rank random matrices]{On the distribution of the entries of a\\ fixed-rank random matrix over a finite field}

\usepackage{amsmath}
\usepackage{amssymb}
\usepackage{amsthm}
\usepackage{bbm}
\usepackage{bm}
\usepackage{geometry}
\usepackage{color}
\usepackage{hyperref}
\hypersetup{colorlinks=true}
\usepackage{enumitem}
\setlist[enumerate]{label=(\roman*),labelindent=1em,itemsep=0.5em,topsep=0.5em}

\newtheorem{theorem}{Theorem}[section]

\newtheorem{lemma}[theorem]{Lemma}
\theoremstyle{remark}
\newtheorem{remark}{Remark}[section]

\DeclareMathOperator{\wt}{wt}
\DeclareMathOperator{\ct}{ct}
\DeclareMathOperator{\rk}{rk}

\allowdisplaybreaks

\begin{document}

\begin{abstract}
    Let $r > 0$ be an integer, let $\mathbb{F}_q$ be a finite field of $q$ elements, and let $\mathcal{A}$ be a nonempty proper subset of $\mathbb{F}_q$.
    Moreover, let $\mathbf{M}$ be a random $m \times n$ rank-$r$ matrix over $\mathbb{F}_q$ taken with uniform distribution.
    We prove, in a precise sense, that, as $m, n \to +\infty$ and $r,q,\mathcal{A}$ are fixed, the number of entries of $\mathbf{M}$ that belong to $\mathcal{A}$ approaches a normal distribution.
\end{abstract}

\maketitle

\section{Introduction}

Let $\mathbb{F}_q$ be a finite field of $q$ element.
For every matrix $\mathbf{M}$ over $\mathbb{F}_q$, let $\wt(\mathbf{M})$ be the \emph{weight} of $\mathbb{F}_q$, that is, the number of nonzero entries of $\mathbf{M}$.

Migler, Morrison, and Ogle~\cite{MR2257016} proved a formula for the expected value of $\wt(\mathbf{M})$ when $\mathbf{M}$ is taken at random, with uniform distribution, from the set of $m \times n$ rank-$r$ matrices over $\mathbb{F}_q$.
Furthermore, they suggested that, as $m, n \to +\infty$ and $r,q$ are fixed, an appropriate scaling of $\wt(\mathbf{M})$ approaches a normal distribution.
Sanna~\cite{MR4531531} proved this last claim for $q = 2$ and assuming that $m / n$ converges to a positive real number.

For every $\mathcal{A} \subseteq \mathbb{F}_q$ and for every matrix $\mathbf{M}$ over $\mathbb{F}_q$, let $\ct_{\mathcal{A}}(\mathbf{M})$ be the number of entries of $\mathbf{M}$ that belong to $\mathcal{A}$.
Moreover, put $\gamma_{\mathcal{A}}(q) := \sum_{a \in \mathcal{A}} \gamma_a(q)$, where $\gamma_0(q) := q^{-1} - 1$ and $\gamma_a(q) := q^{-1}$ for each $a \in \mathbb{F}_q^*$, and let
\begin{align*}
    \mu_{\mathcal{A}}(q, m, n) &:= \big(|\mathcal{A}|q^{-1} - \gamma_{\mathcal{A}}(q) q^{-r}\big)mn , \\
    \sigma_{\mathcal{A}}^2(q, m, n) &:= \gamma_{\mathcal{A}}(q)^2 q^{-r}(1 - q^{-r})(m + n)mn ,
\end{align*}
for all integers $m, n > 0$.
Note that $\gamma_{\mathcal{A}}(q) \neq 0$ unless $\mathcal{A} = \varnothing$ or $\mathcal{A} = \mathbb{F}_q$.

Our result is the following.

\begin{theorem}\label{thm:main}
    Fix an integer $r > 0$ and a nonempty set $\mathcal{A} \subsetneq \mathbb{F}_q$.
    Let $\mathbf{M}$ be taken at random, with uniform distribution, from the set of $m \times n$ rank-$r$ matrices over $\mathbb{F}_q$.
    Then, as $m, n \to +\infty$, we have that
    \begin{equation}\label{equ:main}
        \frac{\raisebox{2pt}{$\ct_{\mathcal{A}}(\mathbf{M}) - \mu_{\mathcal{A}}(q, m, n)$}}{\raisebox{-6pt}{$\sqrt{\strut\sigma_{\mathcal{A}}^2(q, m, n)}$}}
    \end{equation}
    converges in distribution to a standard normal random variable.
\end{theorem}

Roughly speaking, Theorem~\ref{thm:main} asserts that, as $m$ and $n$ both grow, $\ct_{\mathcal{A}}(\mathbf{M})$ approaches a normal random variable with expected value $\mu_{\mathcal{A}}(q, m, n)$ and variance $\sigma_{\mathcal{A}}^2(q, m, n)$.
Note that, if the condition on the rank is dropped, that is, if $\mathbf{M}$ is taken at random with uniform distribution from the set of $m \times n$ matrices over $\mathbb{F}_q$, then an easy application of the central limit theorem yields that $\ct_{\mathcal{A}}(\mathbf{M})$ approaches a normal random variable with expected value $|\mathcal{A}|q^{-1} mn$ and variance $|\mathcal{A}|q^{-1}\left(1 - |\mathcal{A}|q^{-1}\right)mn$.

Before we proceed, let us outline the main ideas of the proof of Theorem~\ref{thm:main}.
First, using full-rank factorization and the well-known formula for the number of $m \times n$ rank-$r$ matrices over $\mathbb{F}_q$, it is shown that, for the sake of proving Theorem~\ref{thm:main}, we can assume that $\mathbf{M} = \mathbf{X} \mathbf{Y}$,
where $\mathbf{X}$ and $\mathbf{Y}$ are $m \times r$ and $r \times n$ independent random matrices taken with uniform distribution from their respective spaces.
Second, the event that the product of a row of $\mathbf{X}$ and a column of $\mathbf{Y}$ is equal to a prescribed element of $\mathbb{F}_q$ is handled via the Fourier transform of $\mathbb{F}_q$ respect to multiplicative characters.
The use of multiplicative characters is necessary to conveniently ``separate'' the entries of $\mathbf{X}$ by the entries of $\mathbf{Y}$ in two factors of a product.
However, it introduces some complications (essentially because the Fourier inversion formula holds only for functions $\mathbb{F}_q^t \to \mathbb{C}$ that are supported on $(\mathbb{F}_q^*)^t$), which are dealt with by a kind of M\"obius transform.
Finally, all of this makes possible to write \eqref{equ:main} as a main term, which converges in distribution to a standard normal random variable, plus an error term, which is shown to be negligible.

It might be interesting to strenghten Theorem~\ref{thm:main} by letting also $r$ goes to infinity, but in a way controlled by $m$ and $n$ (see Remark~\ref{rmk:r-to-oo}).

\section{General notations and definitions}

For every finite set $\mathcal{A}$, we let $|\mathcal{A}|$ be the number of elements of $\mathcal{A}$.
For each statement $S$, we let $\mathbbm{1}[S]$ be equal to $1$ if $S$ is true, and to $0$ if $S$ is false.
For every event $E$, we let $\mathbb{P}[E]$ be the probability that $E$ occurs.
For each real or complex random variable $X$, we write $\mathbb{E}[X]$ and $\mathbb{V}[X]$ for the expected value and the variance of $X$.
For every sequence $(X_n)$ of random variables, we write $X_n \xrightarrow{d} X$ to denote that $(X_n)$ converges in distribution to $X$.
For a complex random variable $Z = X + \bm{i}Y$, where $X$ and $Y$ are real random variables and $\bm{i}$ is the imaginary unity, the \emph{covariance matrix} of $Z$ is the covariance matrix of the random vector $(X, Y)$.
Also, we say that $Z$ is a \emph{complex normal random variable} if the random vector $(X, Y)$ follows a bivariate normal distribution.
For each integer $r > 0$, we set $[r] := \{1, \dots, r\}$.
We say that a function $f \colon \mathcal{X} \to \mathbb{C}$ is \emph{supported} on a set $\mathcal{Y}$ if $f(x) = 0$ for every $x \in \mathcal{X} \setminus \mathcal{Y}$.
We adopt the usual convention that the empty sum and the empty product are equal to $0$ and $1$, respectively.

\section{Preliminaries on the Fourier transform}\label{sec:prelim-fourier}

\subsection{Characters of finite fields}\label{sec:characters}

We recall some basics facts about characters of finite fields~(see, e.g., \cite[Chapter 5, Section~1]{MR1294139} and \cite[Chapter~10, Section~1]{MR3087321}).

Given a finite abelian group $G$, a \emph{character} of $G$ is a group homomorphism $G \to \mathbb{C}^*$.
The set of characters of $G$ is denoted by $\widehat{G}$ and is a finite abelian group respect to the pointwise product of functions.
The identity of $\widehat{G}$ is the \emph{trivial character}, which sends each element of $G$ to~$1$, while the inverse of each $\chi \in \widehat{G}$ is the pointwise complex conjugation of $\chi$, which is denoted~by~$\overline{\chi}$.

The \emph{additive characters} of $\mathbb{F}_q$ are the characters of $\mathbb{F}_q$ as an additive group.
We let $\psi_0$ denote the trivial additive character of $\mathbb{F}_q$.
The \emph{multiplicative characters} of $\mathbb{F}_q$ are the characters of $\mathbb{F}_q^*$ as a multiplicative group.
We let $\chi_0$ denote the trivial multiplicative character of $\mathbb{F}_q$.
Moreover, we extend each multiplicative character $\chi$ of $\mathbb{F}_q$ to a function $\mathbb{F}_q \to \mathbb{C}$ by setting $\chi(0) := 0$.

The additive and multiplicative characters of $\mathbb{F}_q$ satisfy the orthogonality relations:

\noindent
\begin{minipage}[l]{0.5\textwidth}
    \begin{equation}\label{equ:ortho-a-sum-psi}
        \frac1{q} \sum_{\psi \in \widehat{\mathbb{F}_q}} \psi(a) = \mathbbm{1}[a = 0]
    \end{equation}
\end{minipage}%
\begin{minipage}[r]{0.5\textwidth}
    \begin{equation}\label{equ:ortho-a-sum-chi}
        \frac1{q - 1} \sum_{\chi \in \widehat{\mathbb{F}_q^*}} \chi(a) = \mathbbm{1}[a = 1]
    \end{equation}
\end{minipage}
\vspace{0.5em}

\noindent
for every $a \in \mathbb{F}_q$, and

\noindent
\begin{minipage}[l]{0.5\textwidth}
    \begin{equation}\label{equ:ortho-psi-sum-a}
        \frac1{q} \sum_{a \in \mathbb{F}_q} \psi(a) = \mathbbm{1}[\psi = \psi_0]
    \end{equation}
\end{minipage}%
\begin{minipage}[r]{0.5\textwidth}
    \begin{equation}\label{equ:ortho-chi-sum-a}
        \frac1{q - 1} \sum_{a \in \mathbb{F}_q} \chi(a) = \mathbbm{1}[\chi = \chi_0]
    \end{equation}
\end{minipage}
\vspace{0.5em}

\noindent
for every $\psi \in \widehat{\mathbb{F}_q}$ and $\chi \in \widehat{\mathbb{F}_q^*}$.

For every function $f \colon \mathbb{F}_q^t \to \mathbb{C}$ that is supported on $(\mathbb{F}_q^*)^t$, the \emph{Fourier transform} of $f$ is the function $\widehat{f} \colon \widehat{\mathbb{F}_q}^t \to \mathbb{C}$ defined by
\begin{equation}\label{equ:def-fourier-transform}
    \widehat{f}(\chi_1, \dots, \chi_t) := \frac1{(q-1)^t}\sum_{a_1, \dots, a_t \in \mathbb{F}_q} f(a_1, \dots, a_t) \overline{\chi_1}(a_1) \cdots \overline{\chi_t}(a_t)
\end{equation}
for every $\chi_1, \dots, \chi_t \in \widehat{\mathbb{F}_q^*}$.\footnote{We normalize the Fourier transform by the factor $(q-1)^{-t}$ because later this simplifies some formulas.}
From the orthogonality relation \eqref{equ:ortho-a-sum-chi}, it easily follows that
\begin{equation}\label{equ:fourier-inversion}
    f(a_1, \dots, a_t) = \sum_{\chi_1, \dots, \chi_t \in \widehat{\mathbb{F}_q^*}} \widehat{f}(\chi_1, \cdots, \chi_t) \chi_1(a_1) \cdots \chi_t(a_t)
\end{equation}
for every $a_1, \dots, a_t \in \mathbb{F}_q$, which is the \emph{Fourier inversion formula}.

\subsection{M\"obius transform}\label{sec:mobius-transform}

We need a kind of \emph{M\"obius transform} and its corresponding inversion formula, which is essentially a consequence of the inclusion-exclusion principle (see, e.g., \cite[Example~3.8.3]{MR2868112}).

First, note that from the binomial theorem it easily follows that
\begin{equation}\label{equ:sum-over-subsets}
    \sum_{\mathcal{A} \subseteq \mathcal{B}} (-1)^{|\mathcal{A}|} = \mathbbm{1}\!\big[\mathcal{B} = \varnothing\big] ,
\end{equation}
for every finite set $\mathcal{B}$.

Throughout the rest of Section~\ref{sec:prelim-fourier}, let $r > 0$ be a fixed integer.
For every $\mathcal{S} \subseteq [r]$, we write $a_{\mathcal{S}}$ to denote the $|\mathcal{S}|$-tuple $(a_{k_1}, \dots, a_{k_{|\mathcal{S}|}})$, where $k_1 < \cdots < k_{|\mathcal{S}|}$ are the elements of $\mathcal{S}$.
(If $\mathcal{S}$ is empty, then $a_{\mathcal{S}}$ is the empty tuple).
Moreover, we write $a_{(\mathcal{S})}$ to denote the $r$-tuple $(b_1, \dots, b_r)$, where $b_k := 0$ if $k \notin \mathcal{S}$, and $b_k := a_k$ if $k \in \mathcal{S}$.

For every function $f \colon \mathbb{F}_q^r \to \mathbb{C}$ and for every $\mathcal{S} \subseteq [r]$, we define the function $f_\mathcal{S} \colon \mathbb{F}_q^{|\mathcal{S}|} \to \mathbb{C}$ by
\begin{equation}\label{equ:def-f-S}
    f_{\mathcal{S}}(a_{\mathcal{S}}) := \sum_{\mathcal{T} \subseteq \mathcal{S}} (-1)^{|\mathcal{S} \setminus \mathcal{T}|} f(a_{(\mathcal{T})}) ,
\end{equation}
for every $a_{\mathcal{S}} \in \mathbb{F}_q^{|\mathcal{S}|}$.

\begin{lemma}\label{lem:mobius-inversion}
    Let $f \colon \mathbb{F}_q^r \to \mathbb{C}$.
    Then, for every $\mathcal{S} \subseteq [r]$, the function $f_{\mathcal{S}}$ is supported on $(\mathbb{F}_q^*)^{|\mathcal{S}|}$.
    Moreover, we have that
    \begin{equation}\label{equ:mobius-inversion}
        f(a_1, \dots, a_r) = \sum_{\mathcal{S} \subseteq [r]} f_{\mathcal{S}}(a_{\mathcal{S}})
    \end{equation}
    for every $a_1, \dots, a_r \in \mathbb{F}_q$.
\end{lemma}

\begin{proof}
    First, let us prove that for every $\mathcal{S} \subseteq [r]$ the function $f_\mathcal{S}$ is supported on $(\mathbb{F}_q^*)^{|\mathcal{S}|}$.
    Pick any $a_{\mathcal{S}} \in \mathbb{F}_q^{|\mathcal{S}|} \setminus (\mathbb{F}_q^*)^{|\mathcal{S}|}$.
    Hence, there exists $k_0 \in \mathcal{S}$ such that $a_{k_0} = 0$.
    Therefore, by~\eqref{equ:def-f-S} we have that
    \begin{align*}
        f_{\mathcal{S}}(a_{\mathcal{S}}) &= \sum_{\mathcal{T} \subseteq \mathcal{S} \setminus \{k_0\}} (-1)^{|\mathcal{S} \setminus \mathcal{T}|} f(a_{(\mathcal{T})}) + \sum_{\{k_0\} \subseteq \mathcal{T} \subseteq \mathcal{S}} (-1)^{|\mathcal{S} \setminus \mathcal{T}|} f(a_{(\mathcal{T})}) \\
        &= \sum_{\mathcal{T} \subseteq \mathcal{S} \setminus \{k_0\}} (-1)^{|\mathcal{S} \setminus \mathcal{T}|} f(a_{(\mathcal{T})}) - \sum_{\mathcal{T}^\prime \subseteq \mathcal{S} \setminus \{k_0\}} (-1)^{|\mathcal{S} \setminus \mathcal{T}^\prime|} f(a_{(\mathcal{T}^\prime)}) = 0 ,
    \end{align*}
    where we used the fact that each set $\mathcal{T}$ satisfying $\{k_0\} \subseteq \mathcal{T} \subseteq \mathcal{S}$ can be written in a unique way as $\mathcal{T} = \mathcal{T}^\prime \cup \{k_0\}$ with $\mathcal{T}^\prime \subseteq \mathcal{S} \setminus \{k_0\}$.
    The claim is proven.

    Let us prove~\eqref{equ:mobius-inversion}.
    From~\eqref{equ:def-f-S} and~\eqref{equ:sum-over-subsets}, we get that
    \begin{align*}
        \sum_{\mathcal{S} \subseteq [r]} f_{\mathcal{S}}(a_{\mathcal{S}})
        &= \sum_{\mathcal{S} \subseteq [r]} \sum_{\mathcal{T} \subseteq\mathcal{S}} (-1)^{|\mathcal{S} \setminus \mathcal{T}|} f(a_{(\mathcal{T})})
        = \sum_{\mathcal{T} \subseteq [r]} f(a_{(\mathcal{T})}) \sum_{\mathcal{T} \subseteq \mathcal{S} \subseteq [r]} (-1)^{|\mathcal{S} \setminus \mathcal{T}|} \\
        &= \sum_{\mathcal{T} \subseteq [r]} f(a_{(\mathcal{T})}) \sum_{\mathcal{S}^\prime \subseteq [r] \setminus \mathcal{T}} (-1)^{|\mathcal{S}^\prime|} = f(a_1, \dots, a_r) ,
    \end{align*}
    where we wrote $\mathcal{S} = \mathcal{S}^\prime \cup \mathcal{T}$.
    The proof is complete.
\end{proof}

\subsection{M\"obius--Fourier inversion formula}

We can combine the results of Sections~\ref{sec:characters} and~\ref{sec:mobius-transform} to obtain a \emph{M\"obius--Fourier inversion formula}.

\begin{lemma}\label{lem:mobius-fourier-inversion}
    Let $f \colon \mathbb{F}_q^r \to \mathbb{C}$.
    Then we have that
    \begin{equation*}
        f(a_1, \dots, a_r) = \sum_{\mathcal{S} \subseteq [r]} \sum_{\chi_{\mathcal{S}} \in \widehat{\mathbb{F}_q^*}^{|\mathcal{S}|}} \widehat{f_{\mathcal{S}}} (\chi_{\mathcal{S}}) \prod_{k \in \mathcal{S}} \chi_k (a_k) ,
    \end{equation*}
    for every $a_1, \dots, a_r \in \mathbb{F}_q$.
\end{lemma}
\begin{proof}
    The claim easily follows from the Fourier inversion formula~\eqref{equ:fourier-inversion} and Lemma~\ref{lem:mobius-inversion}.
\end{proof}

For every function $f \colon \mathbb{F}_q^r \to \mathbb{C}$ and for every $\mathcal{S} \subseteq [r]$, let $f_{(\mathcal{S})} \colon \mathbb{F}_q^{|\mathcal{S}|} \to \mathbb{C}$ be the function defined by $f_{(\mathcal{S})}(a_\mathcal{S}) := f(a_{(\mathcal{S})})$ for each $a_\mathcal{S} \in \mathbb{F}_q^{|\mathcal{S}|}$.

\begin{lemma}\label{lem:fourier-transform-of-f-S}
    Let $f \colon \mathbb{F}_q^r \to \mathbb{C}$ and $\mathcal{S} \subseteq [r]$.
    Then we have that
    \begin{equation*}
        \widehat{f_{\mathcal{S}}}(\chi_{\mathcal{S}}) = \sum_{\{\chi_k \neq \chi_0 : k \in \mathcal{S}\} \subseteq \mathcal{T} \subseteq \mathcal{S}} (-1)^{|\mathcal{S} \setminus \mathcal{T}|}\widehat{f_{(\mathcal{T})}}(\chi_{\mathcal{T}}) ,
    \end{equation*}
    for every $\chi_{\mathcal{S}} \in \widehat{\mathbb{F}_q^*}^{|\mathcal{S}|}$.
\end{lemma}
\begin{proof}
    From~\eqref{equ:def-fourier-transform} and~\eqref{equ:def-f-S}, we get that
    \begin{align}\label{equ:fourier-transform-of-f-S-1}
        \widehat{f_{\mathcal{S}}}(\chi_{\mathcal{S}})
        &= \frac1{(q-1)^{|\mathcal{S}|}} \sum_{a_{\mathcal{S}} \in \mathbb{F}_q^{|\mathcal{S}|}} f_{\mathcal{S}}(a_{\mathcal{S}}) \prod_{k \in \mathcal{S}} \overline{\chi_k}(a_k) \\
        &= \frac1{(q-1)^{|\mathcal{S}|}} \sum_{a_{\mathcal{S}} \in \mathbb{F}_q^{|\mathcal{S}|}} \sum_{\mathcal{T} \subseteq \mathcal{S}} (-1)^{|\mathcal{S} \setminus \mathcal{T}|} f(a_{(\mathcal{T})}) \prod_{k \in \mathcal{S}} \overline{\chi_k}(a_k) \nonumber\\
        &= \sum_{\mathcal{T} \subseteq \mathcal{S}} (-1)^{|\mathcal{S} \setminus \mathcal{T}|} \left( \frac1{(q-1)^{|\mathcal{T}|}}\sum_{a_{\mathcal{T}} \in \mathbb{F}_q^{|\mathcal{T}|}} f_{(\mathcal{T})}(a_{\mathcal{T}}) \prod_{k \in \mathcal{T}} \overline{\chi_k}(a_k) \right) \nonumber\\
        &\phantom{\hspace{10em}}\cdot \left( \frac1{(q-1)^{|\mathcal{S} \setminus \mathcal{T}|}} \sum_{a_{\mathcal{S} \setminus \mathcal{T}} \in \mathbb{F}_q^{|\mathcal{S} \setminus \mathcal{T}|}} \prod_{k \in \mathcal{S} \setminus \mathcal{T}} \overline{\chi_k}(a_k) \right) . \nonumber
    \end{align}
    Furthermore, for every $\mathcal{U} \subseteq [r]$, we have that
    \begin{equation}\label{equ:fourier-transform-of-f-S-2}
        \sum_{a_{U} \in \mathbb{F}_q^{|\mathcal{U}|}} \prod_{k \in \mathcal{U}} \overline{\chi_k}(a_k) = \prod_{k \in \mathcal{U}} \left(\sum_{a \in \mathbb{F}_q} \overline{\chi_k}(a)\right) = (q - 1)^{|\mathcal{U}|} \, \mathbbm{1}[\chi_k = \chi_0 \text{ for each } k \in \mathcal{U}] ,
    \end{equation}
    where we employed the orthogonality relation~\eqref{equ:ortho-chi-sum-a}.
    At this point, the claim follows by combining~\eqref{equ:def-fourier-transform}, \eqref{equ:fourier-transform-of-f-S-1}, and \eqref{equ:fourier-transform-of-f-S-2}.
\end{proof}

\subsection{Some Fourier transforms}

For every $a \in \mathbb{F}_q$, define the function $f^{(a)} \colon \mathbb{F}_q^r \to \mathbb{C}$ by
\begin{equation*}
    f^{(a)}(a_1, \dots, a_r) = \mathbbm{1}\!\!\left[\,\sum_{k = 1}^r a_k = a\right] ,
\end{equation*}
for every $a_1, \dots, a_r \in \mathbb{F}_q^r$.

Furthermore, let $\chi_{\bm{0}}$ denote a tuple $(\chi_0, \dots, \chi_0)$, where the length will be always clear from the context.

\begin{lemma}\label{lem:fourier-transform-of-fz-T}
    For every $a \in \mathbb{F}_q$ and $\mathcal{T} \subseteq [r]$, we have that
    \begin{equation*}
        \widehat{f_{(\mathcal{T})}^{(a)}}(\chi_{\bm{0}}) = \frac1{q} - \frac{\gamma_a(q)}{(1 - q)^{|\mathcal{T}|}} .
    \end{equation*}
\end{lemma}
\begin{proof}
    This is essentially the evaluation of a generalized Jacobi sum of trivial characters, which is a well-known subject (see, e.g., \cite[Theorem~6.1.35]{MR3087321}), but we include the details for completeness.

    First, from~\eqref{equ:ortho-psi-sum-a} it follows that
    \begin{equation}\label{equ:fourier-transform-of-fz-T-1}
        \sum_{a_{\mathcal{T}} \in (\mathbb{F}_q^*)^{|\mathcal{T}|}} \, \prod_{k \in \mathcal{T}} \psi(a_k)
        = \prod_{k \in \mathcal{T}} \left( \sum_{a \in \mathbb{F}_q^*} \psi(a) \right)
        = \begin{cases}
            (q - 1)^{|\mathcal{T}|} & \text{ if } \psi = \psi_0 ; \\
            (-1)^{|\mathcal{T}|} & \text{ if } \psi \neq \psi_0 ;
        \end{cases}
    \end{equation}
    for every $\psi \in \widehat{\mathbb{F}_q}$.
    Then, from~\eqref{equ:def-fourier-transform}, ~\eqref{equ:ortho-a-sum-psi}, and~\eqref{equ:fourier-transform-of-fz-T-1}, we get that
    \begin{align*}
        \widehat{f_{(\mathcal{T})}^{(a)}}(\chi_{\bm{0}})
        &= \frac1{(q - 1)^{|\mathcal{T}|}} \sum_{a_{\mathcal{T}} \in \mathbb{F}_q^{|\mathcal{T}|}} f_{(\mathcal{T})}^{(a)}(a_{\mathcal{T}}) \prod_{k \in \mathcal{T}} \overline{\chi_0}(a_k)
        = \frac1{(q - 1)^{|\mathcal{T}|}} \sum_{a_{\mathcal{T}} \in (\mathbb{F}_q^*)^{|\mathcal{T}|}} f_{(\mathcal{T})}^{(a)}(a_{\mathcal{T}}) \\
        &= \frac1{(q - 1)^{|\mathcal{T}|}} \sum_{a_{\mathcal{T}} \in (\mathbb{F}_q^*)^{|\mathcal{T}|}} \mathbbm{1}\!\!\left[\,\sum_{k \in \mathcal{T}} a_k = a\right]
        = \frac1{(q - 1)^{|\mathcal{T}|}} \sum_{a_{\mathcal{T}} \in (\mathbb{F}_q^*)^{|\mathcal{T}|}} \frac1{q} \sum_{\psi \in \widehat{\mathbb{F}_q}} \psi\!\left(\,\sum_{k \in \mathcal{T}} a_k - a\right) \\
        &= \frac1{q(q - 1)^{|\mathcal{T}|}} \sum_{\psi \in \widehat{\mathbb{F}_q}} \, \sum_{a_{\mathcal{T}} \in (\mathbb{F}_q^*)^{|\mathcal{T}|}} \, \prod_{k \in \mathcal{T}} \psi(a_k) \, \overline{\psi}(a)
        = \frac1{q} + \frac1{q(1 - q)^{|\mathcal{T}|}} \sum_{\psi \in \widehat{\mathbb{F}_q} \setminus \{\psi_0\}} \overline{\psi}(a) \\
        &= \frac1{q} + \frac1{q(1 - q)^{|\mathcal{T}|}} \! \left(\sum_{\psi \in \widehat{\mathbb{F}_q}} \overline{\psi}(a) - 1 \right)
        = \frac1{q} - \frac{\gamma_a(q)}{(1 - q)^{|\mathcal{T}|}} ,
    \end{align*}
    since $\gamma_a(q) = q^{-1} - \mathbbm{1}[a = 0]$.
    The proof is complete.
\end{proof}

\begin{lemma}\label{lem:fourier-transform-of-fz-S}
    For every $a \in \mathbb{F}_q$ and $\mathcal{S} \subseteq [r]$, we have that
    \begin{equation*}
        \widehat{f_{\mathcal{S}}^{(a)}}(\chi_{\bm{0}}) = \frac{\mathbbm{1}[\mathcal{S} = \varnothing]}{q} - \gamma_a(q) \left(\frac1{q} - 1\right)^{-|\mathcal{S}|} .
    \end{equation*}
\end{lemma}
\begin{proof}
    By Lemma~\ref{lem:fourier-transform-of-f-S} and Lemma~\ref{lem:fourier-transform-of-fz-T}, we have that
    \begin{align*}
        \widehat{f^{(a)}_{\mathcal{S}}}(\chi_{\bm{0}})
        &= \sum_{\mathcal{T} \subseteq \mathcal{S}} (-1)^{|\mathcal{S} \setminus \mathcal{T}|}\widehat{f^{(a)}_{(\mathcal{T})}}(\chi_{\bm{0}})
        = \sum_{\mathcal{T} \subseteq \mathcal{S}} (-1)^{|\mathcal{S} \setminus \mathcal{T}|}\left(\frac1{q} - \frac{\gamma_a(q)}{(1 - q)^{|\mathcal{T}|}}\right) \\
        &= \frac1{q} \sum_{\mathcal{T} \subseteq \mathcal{S}} (-1)^{|\mathcal{S} \setminus \mathcal{T}|} - \gamma_a(q) \sum_{\mathcal{T} \subseteq \mathcal{S}} (-1)^{|\mathcal{S} \setminus \mathcal{T}|} (1 - q)^{-|\mathcal{T}|} \\
        &= \frac{\mathbbm{1}[\mathcal{S} = \varnothing]}{q} - \gamma_a(q) \left(\frac1{q} - 1\right)^{-|\mathcal{S}|} ,
    \end{align*}
    where we used~\eqref{equ:sum-over-subsets} and the more general fact that
    \begin{equation*}
        \sum_{\mathcal{A} \subseteq \mathcal{B}} s^{|\mathcal{B} \setminus \mathcal{A}|} \, t^{|\mathcal{A}|} = (s + t)^{|\mathcal{B}|}
    \end{equation*}
    for every finite set $\mathcal{B}$ and for all real numbers $s$ and $t$.
\end{proof}

\section{Further preliminaries}

For every field $\mathbb{K}$, let $\mathbb{K}^{m \times n}$ be the vector space of $m \times n$ matrices over $\mathbb{K}$, and let $\mathbb{K}^{m \times n, r}$ be the set of $m \times n$ rank-$r$ matrices over $\mathbb{K}$.
The next lemma regards the \emph{full-rank factorization} of matrices and it is well known~(cf.~\cite[Theorem~2]{MR1573394}).

\begin{lemma}\label{lem:full-rank-factorization}
    Let $\mathbb{K}$ be a field.
    For every $\mathbf{N} \in \mathbb{K}^{m \times n, r}$ there exist $\mathbf{X}_0 \in \mathbb{K}^{m \times r, r}$ and $\mathbf{Y}_0 \in \mathbb{K}^{r \times n, r}$ such that $\mathbf{N} = \mathbf{X}_0 \mathbf{Y}_0$.
    Moreover, if $\mathbf{N} = \mathbf{X} \mathbf{Y}$ for some $\mathbf{X} \in \mathbb{K}^{m \times r}$ and $\mathbf{Y}  \in \mathbb{K}^{r \times n}$, then there exists $\mathbf{R} \in \mathbb{K}^{r \times r, r}$ such that $\mathbf{X} = \mathbf{X}_0 \mathbf{R}$ and $\mathbf{Y} = \mathbf{R}^{-1}\mathbf{Y}_0$.
\end{lemma}
\begin{proof}
    See, e.g., \cite[Lemma~2.1]{MR4531531}.
    There the second part of the lemma is stated with $\mathbf{X} \in \mathbb{K}^{m \times r,r}$ and $\mathbf{Y}  \in \mathbb{K}^{r \times n,r}$ instead of $\mathbf{X} \in \mathbb{K}^{m \times r}$ and $\mathbf{Y}  \in \mathbb{K}^{r \times n}$.
    However, if $\mathbf{X} \in \mathbb{K}^{m \times r}$ and $\mathbf{Y}  \in \mathbb{K}^{r \times n}$ satisfy $\mathbf{X}\mathbf{Y} \in \mathbb{K}^{m \times n, r}$, then $\mathbf{X} \in \mathbb{K}^{m \times r,r}$ and $\mathbf{Y}  \in \mathbb{K}^{r \times n,r}$.
    Therefore, the two versions are equivalent.
\end{proof}

\begin{lemma}\label{lem:same-as-XY}
    Let $\mathbf{M} \in \mathbb{F}_q^{m \times n, r}$, $\mathbf{X} \in \mathbb{F}_q^{m \times r}$, and $\mathbf{Y} \in \mathbb{F}_q^{r \times n}$ be independent random matrices uniformly distributed in their respective spaces.
    Then we have that
    \begin{equation}\label{equ:same-as-XY-sum}
        \sum_{\mathbf{N} \in \mathbb{F}_q^{m \times n}} \left|\mathbb{P}\big[\mathbf{X}\mathbf{Y} = \mathbf{N}\big] - \mathbb{P}\big[\mathbf{M} = \mathbf{N}\big]\!\right| \to 0 ,
    \end{equation}
    as $m, n \to +\infty$ and $r$ is fixed.
\end{lemma}
\begin{proof}
	It is well-known (see, e.g., \cite[Formula~3]{MR2257016}) that
	\begin{equation}\label{equ:rank-count}
		|\mathbb{F}_q^{s \times t, r}| = \prod_{i = 0}^{r - 1} \frac{(q^s - q^i)(q^t - q^i)}{q^r - q^i} ,
	\end{equation}
	for all integers $s, t, r > 0$ with $r \leq \min(s, t)$.

    Furthermore, we have that
	\begin{equation}\label{equ:q-limit}
		\frac{\prod_{i=0}^{r - 1} (q^m - q^i)(q^n - q^i)}{q^{mr} \cdot q^{rn}}
		= \prod_{i=0}^{r - 1} \frac{(q^m - q^i)(q^n - q^i)}{q^m \cdot q^n}
		= \prod_{i=0}^{r - 1} (1 - q^{i-m})(1 - q^{i-n}) \to 1 .
	\end{equation}
	as $m, n \to +\infty$ and $r$ is fixed.

	Let us split the sum in~\eqref{equ:same-as-XY-sum} into three sums $\Sigma_{(<)}$, $\Sigma_{(=)}$, $\Sigma_{(>)}$ according to the rank of $\mathbf{N}$ being less than, equal to, or greater than $r$, respectively.
    We have to prove that, in the aforementioned limit, each of these sums goes to zero.

    For every matrix $\mathbf{Z}$ over $\mathbb{F}_q$, let $\rk(\mathbf{Z})$ denote the rank of $\mathbf{Z}$.
	From~\eqref{equ:rank-count} and~\eqref{equ:q-limit}, we get that
	\begin{align*}
		\Sigma_{(<)} &= \sum_{\substack{\mathbf{N} \in \mathbb{F}_q^{m \times n} \\ \rk(\mathbf{N}) < r}} \mathbb{P}\big[\mathbf{X}\mathbf{Y} = \mathbf{N}\big]
        = \mathbb{P}\big[\!\rk(\mathbf{X}\mathbf{Y}) < r\big]
        = 1 - \mathbb{P}[\mathbf{X} \in \mathbb{F}_q^{m \times r, r}] \, \mathbb{P}[\mathbf{Y} \in \mathbb{F}_q^{r \times n, r}] \\
		&= 1 - \frac{|\mathbb{F}_q^{m \times r, r}||\mathbb{F}_q^{r \times n, r}|}{|\mathbb{F}_q^{m \times r}||\mathbb{F}_q^{r \times n}|}
		= 1 - \frac{\prod_{i=0}^{r - 1} (q^m - q^i)(q^n - q^i)}{q^{mr} \cdot q^{rn}} \to 0 ,
	\end{align*}
	where we used the fact that $\rk(\mathbf{X}\mathbf{Y}) \leq r$ with equality if and only if $\rk(\mathbf{X}) = \rk(\mathbf{Y}) = r$.

	If $\mathbf{N} \in \mathbb{F}_q^{m \times n, r}$ then, by Lemma~\ref{lem:full-rank-factorization}, there exist matrices $\mathbf{X}_0 \in \mathbb{F}_q^{m \times r, r}$ and $\mathbf{Y}_0 \in \mathbb{F}_q^{r \times n, r}$ such that $\mathbf{N} = \mathbf{X}_0 \mathbf{Y}_0$.
	Moreover, again by Lemma~\ref{lem:full-rank-factorization}, we have that $\mathbf{X}\mathbf{Y} = \mathbf{N}$ if and only if there exists $\mathbf{R} \in \mathbb{F}_q^{r \times r, r}$ such that $\mathbf{X} = \mathbf{X}_0 \mathbf{R}$ and $\mathbf{Y} = \mathbf{R}^{-1} \mathbf{Y}_0$.
	Consequently, we have that
    \begin{equation*}
		\mathbb{P}\big[\mathbf{X}\mathbf{Y} = \mathbf{N}\big]
		= \sum_{\mathbf{R} \in \mathbb{F}_q^{r \times r, r}} \mathbb{P}\big[\mathbf{X} = \mathbf{X}_0\mathbf{R}\big] \, \mathbb{P}\big[\mathbf{Y} = \mathbf{R}^{-1}\mathbf{Y}_0\big] = \frac{|\mathbb{F}_q^{r \times r, r}|}{|\mathbb{F}_q^{m \times r}||\mathbb{F}_q^{r \times n}|} .
	\end{equation*}
	Therefore, we get that
	\begin{align*}
		\Sigma_{(=)} &= \sum_{\mathbf{N} \in \mathbb{F}_q^{m \times n, r}} \left|\mathbb{P}\big[\mathbf{X}\mathbf{Y} = \mathbf{N}\big] - \mathbb{P}\big[\mathbf{M} = \mathbf{N}\big]\!\right|
		= \sum_{\mathbf{N} \in \mathbb{F}_q^{m \times n, r}} \left|\frac{|\mathbb{F}_q^{r \times r, r}|}{|\mathbb{F}_q^{m \times r}||\mathbb{F}_q^{r \times n}|} - \frac1{|\mathbb{F}_q^{m \times n, r}|}\!\right| \\
		&= \left|\frac{|\mathbb{F}_q^{r \times r, r}||\mathbb{F}_q^{m \times n, r}|}{|\mathbb{F}_q^{m \times r}||\mathbb{F}_q^{r \times n}|} - 1 \right|
		= \left|\frac{\prod_{i=0}^{r - 1} (q^m - q^i)(q^n - q^i)}{q^{mr} \cdot q^{rn}} - 1\right| \to 0 ,
	\end{align*}
	where we employed~\eqref{equ:rank-count} and~\eqref{equ:q-limit}.

	Finally, since $\mathbf{X}\mathbf{Y}$ and $\mathbf{M}$ have rank not exceeding $r$, it follows that $\Sigma_{(>)} = 0$.
    Thus all the three sums go to zero and the proof is complete.
\end{proof}

The next result is a version of Slutsky's lemma (cf.~\cite[Lemma~2.8]{MR1652247}).

\begin{lemma}\label{lem:slutsky}
    Let $(U_n)$ and $(V_n)$ be sequences of complex random variables such that $U_n \xrightarrow{d} U$ and $V_n \xrightarrow{d} c$ as $n \to +\infty$, where $U$ is a random variable and $c$ is a constant.
    Then we have that:
    \begin{enumerate}
        \item\label{ite:slutsky:i} $U_n + V_n \xrightarrow{d} U + c$; and
        \item\label{ite:slutsky:ii} $U_n V_n \xrightarrow{d} U c$;
    \end{enumerate}
    as $n \to +\infty$.
\end{lemma}
\begin{proof}
    In \cite[Lemma~2.8]{MR1652247} the result is stated for real random variables.
    However, the proof can be easily adapted by identifying $\mathbb{C}$ with $\mathbb{R}^2$ and applying \cite[Theorem~2.7]{MR1652247} accordingly; noting that, with this identification, the addition and the multiplication of two complex numbers are continuous functions $\mathbb{R}^2 \times \mathbb{R}^2 \to \mathbb{R}^2$.
\end{proof}

\begin{lemma}\label{lem:linear-comb-of-indep-normals}
    Let $c_1, c_2$ be real numbers, and let $N_1,N_2$ be independent normal random variables of expected values $\mu_1,\mu_2$ and variances $\sigma_1^2, \sigma_2^2$, respectively.
    Then $c_1 N_1 + c_2 N_2$ is a normal random variable of expected value $c_1 \mu_1 + c_2 \mu_2$ and variance $c_1^2 \sigma_1^2 + c_2^2 \sigma_2^2$.
\end{lemma}
\begin{proof}
    This fact is well known (cf.~\cite[Exercise~2.1.9]{MR2906465}).
\end{proof}

\section{Proof of Theorem~\ref{thm:main}}

Let $m, n, r > 0$ be integers with $r \leq \min(m, n)$.
Let $\mathbf{X} \in \mathbb{F}_q^{m \times r}$ and $\mathbf{Y} \in \mathbb{F}_q^{r \times n}$ be independent random matrices taken with uniform distribution from their respective spaces.

For every $\mathcal{S} \subseteq [r]$ and $\chi_{\mathcal{S}} \in \widehat{\mathbb{F}_q^*}^{|\mathcal{S}|}$, define the complex random variables
\begin{equation}\label{equ:def-X-Y}
    X_{\mathcal{S}, \chi} := \sum_{i = 1}^m \prod_{k \in \mathcal{S}} \chi_k(x_{i,k})
    \quad\text{and}\quad
    Y_{\mathcal{S}, \chi} := \sum_{j = 1}^n \prod_{k \in \mathcal{S}} \chi_k(y_{k,j}) ,
\end{equation}
and also the real random variables
\begin{equation*}
    Z := \sum_{i = 1}^m \prod_{k = 1}^r \big(1 - \chi_0(x_{i,k})\big)
    \quad\text{and}\quad
    W := \sum_{i = 1}^n \prod_{k = 1}^r \big(1 - \chi_0(y_{k,j})\big) ,
\end{equation*}
where $x_{i,j}$ and $y_{i,j}$ denote the entries of $\mathbf{X}$ and $\mathbf{Y}$, respectively.

The next two lemmas provide the expected values of $X_{\mathcal{S}, \chi}$ and $Y_{\mathcal{S}, \chi}$, and the expected values and the variances of $Z$ and $W$.

\begin{lemma}\label{lem:E-and-V-of-X-Y}
    For all $\mathcal{S} \subseteq [r]$ and $\chi_{\mathcal{S}} \in \widehat{\mathbb{F}_q^*}^{|\mathcal{S}|}$, we have that
    \begin{equation*}
        \mathbb{E}[X_{\mathcal{S}, \chi}] = C(\chi_{\mathcal{S}}) \left(1 - \frac1{q}\right)^{|\mathcal{S}|} m
        \quad\text{and}\quad
        \mathbb{E}[Y_{\mathcal{S}, \chi}] = C(\chi_{\mathcal{S}}) \left(1 - \frac1{q}\right)^{|\mathcal{S}|} n ,
    \end{equation*}
    where
    \begin{equation*}
        C(\chi_{\mathcal{S}}) := \mathbbm{1}[\,\chi_k = \chi_0 \text{ for each } k \in \mathcal{S}\,] .
    \end{equation*}
\end{lemma}
\begin{proof}
    Fix $\chi \in \widehat{\mathbb{F}_q^*}$ and let $c \in \mathbb{F}_q$ be taken at random with uniform distribution.
    From~\eqref{equ:ortho-chi-sum-a} it follows that
    \begin{equation*}
        \mathbb{E}[\chi(c)] = \frac1{q} \sum_{a \in \mathbb{F}_q} \chi(a) = \left(1 - \frac1{q}\right) \mathbbm{1}[\chi = \chi_0] .
    \end{equation*}
    Consequently, if $c_{\mathcal{S}} \in \mathbb{F}_q^{|\mathcal{S}|}$ is a random tuple taken with uniform distribution, then
    \begin{equation*}
        \mathbb{E}\!\left[\,\prod_{k \in \mathcal{S}} \chi_k(c_k) \right] = \prod_{k \in \mathcal{S}} \mathbb{E}\!\big[\chi_k(c_k)\big] = C(\chi_{\mathcal{S}}) \left(1 - \frac1{q}\right)^{|\mathcal{S}|} .
    \end{equation*}
    At this point, the formulas for the expected values of $X_{\mathcal{S},\chi}$ and $Y_{\mathcal{S},\chi}$ follow by linearity.
\end{proof}

\begin{lemma}\label{lem:E-and-V-of-Z-W}
    We have that
    \begin{equation*}
        \mathbb{E}[Z] = \frac1{q^r} m, \quad \mathbb{V}[Z] = \frac1{q^r}\! \left(1 - \frac1{q^r}\right) m, \quad
        \mathbb{E}[W] = \frac1{q^r} n, \quad \mathbb{V}[W] = \frac1{q^r}\! \left(1 - \frac1{q^r}\right) n .
    \end{equation*}
\end{lemma}
\begin{proof}
    The claim follows easily by noticing that $Z$ and $W$ are binomial random variables of $m$ and $n$ trials, respectively, and probability of success equal to $q^{-r}$.
\end{proof}

We can now prove a formula for $\ct_{\mathcal{A}}(\mathbf{X}\mathbf{Y})$, for every $\mathcal{A} \subseteq \mathbb{F}_q$.

\begin{lemma}\label{lem:other-sums-for-Z-W}
    For every $a \in \mathbb{F}_q$, we have that
    \begin{align}
        \label{equ:other-sum-for-Z}
        \sum_{\mathcal{S} \subseteq [r]} \sum_{\chi_{\mathcal{S}} \in \widehat{\mathbb{F}_q^*}^{|\mathcal{S}|}} \widehat{f_{\mathcal{S}}^{(a)}}(\chi_{\mathcal{S}}) \mathbb{E}[Y_{\mathcal{S}, \chi}] X_{\mathcal{S}, \chi} &= \frac1{q} mn - \gamma_a(q) n Z , \\
        \label{equ:other-sum-for-W}
        \sum_{\mathcal{S} \subseteq [r]} \sum_{\chi_{\mathcal{S}} \in   \widehat{\mathbb{F}_q^*}^{|\mathcal{S}|}} \widehat{f_{\mathcal{S}}^{(a)}}(\chi_{\mathcal{S}}) \mathbb{E}[X_{\mathcal{S}, \chi}] Y_{\mathcal{S}, \chi} &= \frac1{q} mn - \gamma_a(q) m W , \\
        \label{equ:other-sum-for-EX-EY}
        \sum_{\mathcal{S} \subseteq [r]} \sum_{\chi_{\mathcal{S}} \in   \widehat{\mathbb{F}_q^*}^{|\mathcal{S}|}} \widehat{f_{\mathcal{S}}^{(a)}}(\chi_{\mathcal{S}}) \mathbb{E}[X_{\mathcal{S}, \chi}] \mathbb{E}[Y_{\mathcal{S}, \chi}] &= \left(\frac1{q} - \frac{\gamma_a(q)}{q^r}  \right) mn .
    \end{align}
\end{lemma}
\begin{proof}
    From Lemma~\ref{lem:E-and-V-of-X-Y} and Lemma~\ref{lem:fourier-transform-of-fz-S}, it follows that
    \begin{align*}
        \sum_{\mathcal{S} \subseteq [r]} & \sum_{\chi_{\mathcal{S}} \in \widehat{\mathbb{F}_q^*}^{|\mathcal{S}|}} \widehat{f_{\mathcal{S}}^{(a)}}(\chi_{\mathcal{S}}) \mathbb{E}[Y_{\mathcal{S}, \chi}] X_{\mathcal{S}, \chi}
        = n \sum_{\mathcal{S} \subseteq [r]} \widehat{f_{\mathcal{S}}^{(a)}}(\chi_{\bm{0}}) \left(1 - \frac1{q}\right)^{|\mathcal{S}|} X_{\mathcal{S}, \chi_{\bm{0}}} \\
        &= n \sum_{\mathcal{S} \subseteq [r]} \left(\frac{\mathbbm{1}[\mathcal{S} = \varnothing]}{q} - \gamma_a(q) \left(\frac1{q} - 1\right)^{-|\mathcal{S}|}\right) \left(1 - \frac1{q}\right)^{|\mathcal{S}|} X_{\mathcal{S}, \chi_{\bm{0}}} \\
        &= \frac1{q} mn - \gamma_a(q) n \sum_{\mathcal{S} \subseteq [r]} (-1)^{|\mathcal{S}|} X_{\mathcal{S}, \chi_{\bm{0}}} ,
    \end{align*}
    since $X_{\varnothing, \chi_{\bm{0}}} = m$.
    Furthermore, from~\eqref{equ:def-X-Y}, we have that
    \begin{align*}
        \sum_{\mathcal{S} \subseteq [r]} (-1)^{|\mathcal{S}|} X_{\mathcal{S}, \chi_{\bm{0}}}
        &= \sum_{\mathcal{S} \subseteq [r]} (-1)^{|\mathcal{S}|} \sum_{i = 1}^m \prod_{k \in \mathcal{S}} \chi_0(x_{i,k})
        = \sum_{i = 1}^m \sum_{\mathcal{S} \subseteq [r]} \prod_{k \in \mathcal{S}} \big({-}\chi_0(x_{i,k})\big) \\
        &= \sum_{i = 1}^m \prod_{k = 1}^r \big(1 - \chi_0(x_{i,k})\big) = Z ,
    \end{align*}
    and~\eqref{equ:other-sum-for-Z} follows.
    The proof of~\eqref{equ:other-sum-for-W} proceeds similarly.

    Finally, taking the expected value of both sides of~\eqref{equ:other-sum-for-Z}, and employing Lemma~\ref{lem:E-and-V-of-Z-W}, we obtain~\eqref{equ:other-sum-for-EX-EY}.
\end{proof}

\begin{lemma}\label{lem:ctz-minus-mean}
    For every $\mathcal{A} \subseteq \mathbb{F}_q$, we have that
    \begin{align*}
        \ct_{\mathcal{A}}(\mathbf{X}\mathbf{Y}) &= \mu_{\mathcal{A}}(q, m, n) + \sum_{\mathcal{S} \subseteq [r]}
        \sum_{\chi_{\mathcal{S}} \in \widehat{\mathbb{F}_q^*}^{|\mathcal{S}|}} \sum_{a \in {\mathcal{A}}} \widehat{f_{\mathcal{S}}^{(a)}} (\chi_{\mathcal{S}}) \big(X_{\mathcal{S}, \chi} - \mathbb{E}[X_{\mathcal{S}, \chi}]\big) \big(Y_{\mathcal{S}, \chi} - \mathbb{E}[Y_{\mathcal{S}, \chi}]\big) \\
        &\phantom{mm}- \gamma_{\mathcal{A}}(q) n \big(Z - \mathbb{E}[Z]\big) - \gamma_{\mathcal{A}}(q) m \big(W - \mathbb{E}[W]\big) .
    \end{align*}
\end{lemma}
\begin{proof}
    Let $a \in \mathbb{F}_q$.
    From Lemma~\ref{lem:mobius-fourier-inversion} and~\eqref{equ:def-X-Y}, we have that
    \begin{align*}
        \ct_{\{a\}}\!(\mathbf{X}\mathbf{Y})
        &= \sum_{i = 1}^m \sum_{j = 1}^n \mathbbm{1}\!\!\left[\,\sum_{k=1}^r x_{i,k} y_{k,j} = a\right]
        = \sum_{i = 1}^m \sum_{j = 1}^n f^{(a)}(x_{i,1} y_{1,j}, \dots, x_{i,r} y_{r,j}) \\
        &= \sum_{i = 1}^m \sum_{j = 1}^n \sum_{\mathcal{S} \subseteq [r]}   \sum_{\chi_{\mathcal{S}} \in \widehat{\mathbb{F}_q^*}^{|\mathcal{S}|}} \widehat{f_{\mathcal{S}}^{(a)}}(\chi_{\mathcal{S}}) \prod_{k \in \mathcal{S}} \chi_k(x_{i,k} y_{k,j}) \\
        &= \sum_{\mathcal{S} \subseteq [r]} \sum_{\chi_{\mathcal{S}} \in    \widehat{\mathbb{F}_q^*}^{|\mathcal{S}|}} \widehat{f_{\mathcal{S}}^{(a)}}(\chi_{\mathcal{S}}) \left( \sum_{i = 1}^m \prod_{k \in \mathcal{S}} \chi_k(x_{i,k}) \right) \left(\sum_{j = 1}^n \prod_{k \in \mathcal{S}} \chi_k(y_{k,j}) \right) \\
        &= \sum_{\mathcal{S} \subseteq [r]} \sum_{\chi_{\mathcal{S}} \in    \widehat{\mathbb{F}_q^*}^{|\mathcal{S}|}} \widehat{f_{\mathcal{S}}^{(a)}}(\chi_{\mathcal{S}}) X_{\mathcal{S},\chi} Y_{\mathcal{S},\chi} .
    \end{align*}
    Then, from the identity
    \begin{align*}
        X_{\mathcal{S}, \chi} Y_{\mathcal{S}, \chi} &= \big(X_{\mathcal{S}, \chi} - \mathbb{E}[X_{\mathcal{S}, \chi}]\big) \big(Y_{\mathcal{S}, \chi} - \mathbb{E}[Y_{\mathcal{S}, \chi}]\big) \\
        &\phantom{mmmm}+ \mathbb{E}[Y_{\mathcal{S}, \chi}] X_{\mathcal{S}, \chi} + \mathbb{E}[X_{\mathcal{S}, \chi}] Y_{\mathcal{S}, \chi} - \mathbb{E}[Y_{\mathcal{S}, \chi}] \mathbb{E}[X_{\mathcal{S}, \chi}] ,
    \end{align*}
    we get that
    \begin{align*}
        \ct_{\{a\}}\!(\mathbf{X}\mathbf{Y}) &= \sum_{\mathcal{S} \subseteq [r]} \sum_{\chi_{\mathcal{S}} \in \widehat{\mathbb{F}_q^*}^{|\mathcal{S}|}} \widehat{f_{\mathcal{S}}^{(a)}}(\chi_{\mathcal{S}}) \big(X_{\mathcal{S}, \chi} - \mathbb{E}[X_{\mathcal{S}, \chi}]\big) \big(Y_{\mathcal{S}, \chi} - \mathbb{E}[Y_{\mathcal{S}, \chi}]\big) \\
        &\phantom{mm}+ \sum_{\mathcal{S} \subseteq [r]} \sum_{\chi_{\mathcal{S}} \in \widehat{\mathbb{F}_q^*}^{|\mathcal{S}|}} \widehat{f_{\mathcal{S}}^{(a)}}(\chi_{\mathcal{S}}) \mathbb{E}[Y_{\mathcal{S}, \chi}] X_{\mathcal{S}, \chi}
        + \sum_{\mathcal{S} \subseteq [r]} \sum_{\chi_{\mathcal{S}} \in \widehat{\mathbb{F}_q^*}^{|\mathcal{S}|}} \widehat{f_{\mathcal{S}}^{(a)}}(\chi_{\mathcal{S}}) \mathbb{E}[X_{\mathcal{S}, \chi}] Y_{\mathcal{S}, \chi} \\
        &\phantom{mm}- \sum_{\mathcal{S} \subseteq [r]} \sum_{\chi_{\mathcal{S}} \in \widehat{\mathbb{F}_q^*}^{|\mathcal{S}|}} \widehat{f_{\mathcal{S}}^{(a)}}(\chi_{\mathcal{S}}) \mathbb{E}[Y_{\mathcal{S}, \chi}] \mathbb{E}[X_{\mathcal{S}, \chi}] .
    \end{align*}
    At this point, the claim follows easily by applying Lemma~\ref{lem:other-sums-for-Z-W} and Lemma~\ref{lem:E-and-V-of-Z-W}, and by summing over all $a \in \mathcal{A}$.
\end{proof}

Fix a nonempty $\mathcal{A} \subsetneq \mathbb{F}_q$ and, for the sake of brevity, let
\begin{equation*}
	\widetilde{\ct}_{\mathcal{A}}(\mathbf{N}) := \frac{\raisebox{2pt}{$\ct_{\mathcal{A}}(\mathbf{N}) - \mu_{\mathcal{A}}(q, m, n)$}}{\raisebox{-6pt}{$\sqrt{\strut\sigma_{\mathcal{A}}^2(q, m, n)}$}}
\end{equation*}
for every $\mathbf{N} \in \mathbb{F}_q^{m \times n}$.
Moreover, hereafter, let $m, n \to +\infty$.

Note that each of the complex random variables $X_{\mathcal{S}, \chi}$ and $Y_{\mathcal{S}, \chi}$ is the sum of independent identically distributed random variables with finite covariance matrices.
Therefore, by the Central Limit Theorem in $\mathbb{R}^2$ (see, e.g., \cite[Theorem~3.9.6]{MR3930614}), we have that $\big(X_{\mathcal{S}, \chi} - \mathbb{E}[X_{\mathcal{S}, \chi}]\big) \!/\! \sqrt{m}$ and $\big(Y_{\mathcal{S}, \chi} - \mathbb{E}[Y_{\mathcal{S}, \chi}]\big) \!/\! \sqrt{n}$ converge in distribution to some complex normal random variables, which we call $X_{\mathcal{S}, \chi}^\prime$ and $Y_{\mathcal{S}, \chi}^\prime$, respectively.

Similarly, each of the real random variables $Z$ and $W$ is the sum of independent identically distributed random variables.
Hence, it follows from the Central Limit Theorem (in $\mathbb{R}$) that $\big(Z - \mathbb{E}[Z]\big) \!/\! \sqrt{\mathbb{V}[Z]}$ and $\big(W - \mathbb{E}[W]\big) \!/\! \sqrt{\mathbb{V}[W]}$ converge in distribution to standard normal random variables, which we call $Z^\prime$ and $W^\prime$, respectively.

From Lemma~\ref{lem:ctz-minus-mean} and Lemma~\ref{lem:E-and-V-of-Z-W}, it follows that
\begin{equation}\label{equ:main1}
	\widetilde{\ct}_{\mathcal{A}}(\mathbf{X}\mathbf{Y})
    = \sum_{\mathcal{S} \subseteq [r]}
    \sum_{\chi_{\mathcal{S}} \in \widehat{\mathbb{F}_q^*}^{|\mathcal{S}|}} \frac{c_{\mathcal{A}, \mathcal{S}, \chi}(q)}{\sqrt{m + n}} \, X_{\mathcal{S}, \chi}^\prime Y_{\mathcal{S}, \chi}^\prime
    - \frac{Z^\prime}{\sqrt{1 + m/n}} - \frac{W^\prime}{\sqrt{1 + n/m}} ,
\end{equation}
where each $c_{\mathcal{A}, \mathcal{S}, \chi}(q)$ depends only on $\mathcal{A}$, $\mathcal{S}$, $\chi_{\mathcal{S}}$, $q$, $r$, and not on $m$ and $n$.

Since $X_{\mathcal{S}, \chi}^\prime$ and $Y_{\mathcal{S}, \chi}^\prime$ are independent, their product converges in distribution to $\widetilde{X}_{\mathcal{S}, \chi}\widetilde{Y}_{\mathcal{S}, \chi}$.
Therefore, from Lemma~\ref{lem:slutsky}\ref{ite:slutsky:ii}, we get that each term of the double sum in~\eqref{equ:main1} converges in distribution to the constant $0$.
Consequently, by Lemma~\ref{lem:slutsky}\ref{ite:slutsky:i}, the double sum in~\eqref{equ:main1} converges in distribution to the constant $0$.

Since $\widetilde{Z}$ and $\widetilde{W}$ are independent, from Lemma~\ref{lem:linear-comb-of-indep-normals} it follows that
\begin{equation*}
    U := -\frac{\widetilde{Z}}{\sqrt{1 + m/n}} - \frac{\widetilde{W}}{\sqrt{1 + n/m}}
\end{equation*}
is a standard normal random variable.

Moreover, from $Z^\prime \xrightarrow{d} \widetilde{Z}$, $W^\prime \xrightarrow{d} \widetilde{W}$, and the fact that $1/\!\sqrt{1 + m/n}$ and $1/\!\sqrt{1 + n/m}$ belong to $(0,1)$, we get easily that
\begin{equation*}
    \frac{\widetilde{Z} - Z^\prime}{\sqrt{1 + m/n}} \xrightarrow{d} 0 \quad\text{and}\quad \frac{\widetilde{W} - W^\prime}{\sqrt{1 + n/m}} \xrightarrow{d} 0 .
\end{equation*}
Therefore, Lemma~\ref{lem:slutsky}\ref{ite:slutsky:i} yields that
\begin{equation*}
    -\frac{Z^\prime}{\sqrt{1 + m/n}} - \frac{W^\prime}{\sqrt{1 + n/m}} = U + \frac{\widetilde{Z} - Z^\prime}{\sqrt{1 + m/n}} + \frac{\widetilde{W} - W^\prime}{\sqrt{1 + n/m}} \xrightarrow{d} U .
\end{equation*}
From a last application of Lemma~\ref{lem:slutsky}\ref{ite:slutsky:i} we get that $\widetilde{\ct}_{\mathcal{A}}(\mathbf{X}\mathbf{Y})$ converges in distribution to $U$.

Let $\mathbf{M}$ be a random matrix taken with uniform distribution from $\mathbb{F}_q^{m \times n, r}$.
Thanks to Lemma~\ref{lem:same-as-XY}, for every real number $t$, we have that
\begin{align*}
	&\big|\mathbb{P}[\widetilde{\ct}_{\mathcal{A}}(\mathbf{M}) \leq t ]
	- \mathbb{P}[\widetilde{\ct}_{\mathcal{A}}(\mathbf{X}\mathbf{Y}) \leq t ] \big|
	= \Big|\sum_{\substack{\mathbf{N} \in \mathbb{F}_q^{m \times n} \\ \widetilde{\ct}_{\mathcal{A}}(\mathbf{N}) \leq t}} \big(\mathbb{P}[\mathbf{M} = \mathbf{N}] - \mathbb{P}[\mathbf{X}\mathbf{Y} = \mathbf{N}]\big) \Big| \\
	&\phantom{mmm}\leq \sum_{\mathbf{N} \in \mathbb{F}_q^{m \times n}} \big|\mathbb{P}[\mathbf{X}\mathbf{Y} = \mathbf{N}] - \mathbb{P}[\mathbf{M} = \mathbf{N}] \big| \to 0 .
\end{align*}
Consequently, we get that $\widetilde{\ct}_{\mathcal{A}}(\mathbf{M})$ and $\widetilde{\ct}_{\mathcal{A}}(\mathbf{X}\mathbf{Y})$ have the same limiting distribution (if it exists).
Since we already proved that $\widetilde{\ct}_{\mathcal{A}}(\mathbf{X}\mathbf{Y})$ converges in distribution to a standard normal random variable, we get that $\widetilde{\ct}_{\mathcal{A}}(\mathbf{M})$ also converges in distribution to a standard normal random variable.

The proof of Theorem~\ref{thm:main} is complete.

\begin{remark}\label{rmk:r-to-oo}
    A crucial part of the proof is the fact that, since $r$ is fixed, the double sum in~\eqref{equ:main1} has a fixed number of terms, and so it is possible to prove that it converges in distribution to the constant $0$ without having to closely inspect its terms.
    If one let $r \to +\infty$, in a way controlled by $m$ and $n$, then it seems likely that understanding the behavior of $\widetilde{\ct}_{\mathcal{A}}(\mathbf{X}\mathbf{Y})$ would require a more detailed study of the terms of the double sum in~\eqref{equ:main1}, since the number of such terms grows with $r$.
\end{remark}

\bibliographystyle{amsplain-no-bysame}
\bibliography{temp}

\end{document}